\documentclass{elsarticle}
\usepackage{amsmath,amssymb,amsfonts,amsthm,graphicx}
\allowdisplaybreaks[2]
\newcommand{\fsu}{\mathcal{F}_{s,\boldsymbol{u}}}

\newcommand{\lblthm}[1]{\label{thm:#1}}
\newcommand{\refthm}[1]{Theorem~\ref{thm:#1}}
\newcommand{\lblprp}[1]{\label{prop:#1}}
\newcommand{\refprp}[1]{Proposition~\ref{prop:#1}}
\newcommand{\lblcor}[1]{\label{cor:#1}}
\newcommand{\refcor}[1]{Corollary~\ref{cor:#1}}
\newcommand{\lbllem}[1]{\label{lem:#1}}
\newcommand{\reflem}[1]{Lemma~\ref{lem:#1}}

\newtheorem{thm}{Theorem}[section]
\newtheorem{dfn}[thm]{Definition}
\newtheorem{prop}[thm]{Proposition}
\newtheorem{lem}[thm]{Lemma}
\newtheorem{cor}[thm]{Corollary}
\newtheorem{rmk}[thm]{Remark}

\renewcommand{\lg}{\log_2}
\newcommand{\abs}[1]{\mathord{\left|#1\right|}}
\newcommand{\serr}[2]{\mathrm{Err}\mathord{\left(#2;#1\right)}}
\newcommand{\aerr}[2]{\mathord{\left|\serr{#1}{#2}\right|}}
\newcommand{\werr}[2]{\mathrm{e}^{\mathrm{wor}}\mathord{\left(#2;#1\right)}}
\newcommand{\wal}[1]{\mathrm{wal}_{#1}}

\newcommand{\intervalsep}{{..}}
\renewcommand{\intervalsep}{,}
\newcommand{\hoi}[2]{\mathord{\left[{#1}\intervalsep{#2}\right)}}
\newcommand{\cli}[2]{\mathord{\left[{#1}\intervalsep{#2}\right]}}

\newcommand{\I}{\hoi01}
\newcommand{\Ic}{\cli01}

\newcommand{\R}{\mathbb R}
\newcommand{\Z}{\mathbb Z}
\newcommand{\B}{\mathbb B}
\renewcommand{\B}{\mathbb F_2}

\newcommand{\N}{\mathbb{N}_0}

\newcommand{\lblsec}[1]{\label{sec:#1}}
\newcommand{\refsec}[1]{Section~\ref{sec:#1}}
\newcommand{\lblrmk}[1]{\label{rmk:#1}}
\newcommand{\refrmk}[1]{Remark~\ref{rmk:#1}}
\newcommand{\lbldfn}[1]{\label{dfn:#1}}
\newcommand{\refdfn}[1]{Definition~\ref{dfn:#1}}
\newcommand{\lbleq}[1]{\label{eq:#1}}
\newcommand{\refeq}[1]{(\ref{eq:#1})}
\newcommand{\lblfig}[1]{\label{fig:#1}}
\newcommand{\reffig}[1]{Figure~\ref{fig:#1}}
\newcommand{\remove}[1]{}

\begin{document}
\begin{frontmatter}
\title{Approximation of Quasi-Monte Carlo worst case error in weighted spaces of infinitely times smooth functions}
\author[M]{Matsumoto Makoto}
\address[M]{Graduate School of Sciences, Hiroshima University, Hiroshima 739-8526 Japan}
\ead{m-mat@math.sci.hiroshima-u.ac.jp}
\author[O]{Ryuichi Ohori}
\address[O]{Fujitsu Laboratories Ltd., Kanagawa 211-8588 Japan}
\ead{ohori.ryuichi@jp.fujitsu.com}
\author[Y]{Takehito Yoshiki\corref{}}
\address[Y]{School of Mathematics and Statistics, The University of New South Wales, Sydney, NSW 2052, Australia}
\ead{takehito.yoshiki1@unsw.edu.au}
\cortext[]{corresponding author}
\begin{abstract}
In this paper, we consider Quasi-Monte Carlo (QMC) worst case error of weighted smooth function classes in $C^\infty\Ic^s$ by a digital net over $\B$.
We show that the ratio of the worst case error to the QMC integration error of an exponential function is bounded above and below by constants.
This result provides us with a simple interpretation that a digital net with small QMC integration error for an exponential function also gives the small integration error for any function in this function space.
\end{abstract}
\end{frontmatter}
\section{Introduction}
\lblsec{introduction}
Quasi-Monte Carlo (QMC) integration is one of methods for numerical integration over the $s$-dimensional unit cube $\I^s$ (see \cite{Dick_Pill}, \cite{Nied} and \cite{Sloan} for details).
We approximate the integral of a function $f\colon \I^s\to\R$
\[
I(f) := \int_{\I^s}f(\boldsymbol{x})\,d\boldsymbol{x}
\]
by a quadrature rule of the form
\[
I_P (f) := \frac1N\sum_{\boldsymbol{x}\in P}f(\boldsymbol{x}),
\]
where $P$ is a point set in $\I^s$ with finite cardinality $N$.

In order to make the integration error ${\serr Pf}:=I_P(f) - I(f)$ small for a class of functions, we often measure the quality of point sets $P\subset \I^s$ using the so-called worst case error.
For a function space $F$ with norm $\|\cdot \|_F$, the worst case error for $F$ by $P$ is defined as the supremum of the absolute value of the integration error in the unit ball of $F$:
\begin{align}
{\werr PF}:=\sup_{\substack{f\in F\\ \| f\|_F\le 1}}{\aerr Pf}.
\end{align}
Apparently, this quantity performs as an upper bound on ${\serr Pf}$ for every $f\in F$:
\begin{align}
{\aerr Pf}\le {\werr PF}\cdot \| f\|_F.
\label{KH}
\end{align}
Thus, our goal is to find a point set $P$ with small value of ${\werr PF}$ for a given function class $F$.
In what follows, we focus on QMC rules using digital nets over $\B$ as point sets (see \refdfn{digital-net}).
This is a special type of construction scheme we often use for practical application of QMC\@.

In the previous works on QMC, the function class consisting of functions $f$ whose derivatives up to order $1$ are continuous has been extensively considered.
For this function class, the star-discrepancy plays an alternative role to the worst case error in (\ref{KH}).
Various types of low discrepancy digital nets $P$ have been developed, whose discrepancy is of order $N^{-1+\epsilon}$ for arbitrary small $\epsilon$ (see \cite{Dick_Pill,Kuiper} for the details).
Recently, smoother function spaces also become targets of research in studies of QMC rules, e.g., $\alpha$-smooth Sobolev space, consisting of functions $f$ whose each derivatives up to order $\alpha$ are continuous for $\alpha\ge 2$.
For this function space, efficient digital nets $P$ also have been established (see, e.g., \cite{Dick_periodic,Dick_nonperiodic}).
They satisfy the higher order convergence of the worst case error $N^{-\alpha+\epsilon}$.

More recently, for some family of functions in $C^\infty\Ic^s$, the existence or theoretical construction algorithm of digital nets $P$ have been developed with the convergence rate of the integration error $N^{-C_s\log N}$ for a constant $C_s$ depending on $s$ \cite{WFupp,WFuppex}.
On the other hand, heuristic algorithm is also used for constructing digital nets giving efficient QMC rules in \cite{MSM}, in which Walsh Figure of Merit (WAFOM) is of great importance.
WAFOM performs as an upper bound on the worst case error of a digital net $P$.
Since WAFOM is computable quickly on computer, we can obtain low-WAFOM point sets by computer search (see, e.g., \cite{MSM,Harase,H}).

As a continuous work, Suzuki~\cite{K:acc} considered more general function spaces
\[
\fsu
:=\left\{f\in C^{\infty}\Ic^d\ \middle|\
\|f\|_{\fsu}:=
\sup_{(\alpha_1,\dots,\alpha_s)\in(\mathbb{N}\cup\{0\})^s}
\frac{
\left\|f^{(\alpha_1,\dots,\alpha_s)}\right\|_{L^1}
}
{\prod_{j=1}^su^{\alpha_j}_j}
<\infty\right\},
\]
with a sequence of positive weights $\boldsymbol{u}=(u_j)_{j\ge 1}$.
(Actually he considered another function space including this smooth function space.)
Here $f^{(\alpha_1,\dots,\alpha_s)}$ is the $(\alpha_1,\dots,\alpha_s)$-th mixed partial derivative of $f$.
When we set $u_j=2$ for all $j$, this space corresponds with the original space considered in the above works.
Suzuki~\cite{K:acc} gave the existence of digital nets which achieve the convergence rate $N^{-C_{s,\boldsymbol{u}}\log N}$ of the worst case error for a constant $C_{s,\boldsymbol{u}}$ depending on $s$ and $\boldsymbol{u}$.
Furthermore, under certain conditions on the weights $\boldsymbol{u}$, this upper bound on the worst case error becomes $C'\cdot N^{-D'(\log N)^{E'}}$ for absolute constants $C',D',E'$, which is a dimension-independent convergence rate.
Computer search algorithm is also effective for finding good point sets in this space as in the above case.
For this function space, we can introduce
 a generalization of WAFOM as an upper bound on the worst case error (see \cite[Remark~6.4]{K:acc}).
For this generalized WAFOM, applying component-by-component (CBC) construction, we can construct digital nets achieving the above dimension-independent convergence rate~\cite{DGSY}.

In the series of the study on QMC for the function space $\fsu$ by a digital net $P$, we have mainly focused on upper bounds on the worst case error.
If we grasp both lower and upper bounds on the worst case error, we could study QMC rules more precisely.
If we expand the function class $\fsu$, such analysis can be obtained.
In fact, a function space $\mathcal{W}_{s,\boldsymbol{a},b}$ including $\fsu$ is established in \cite{K:acc}.
The worst case error for $\mathcal{W}_{s,\boldsymbol{a},b}$ by a digital net $P$ is easily obtained by the definition, which equals a generalized WAFOM of $P$\@.
However, it also contains many other discontinuous functions although we are often interested in smooth functions in $\fsu$.

In this paper, we give feasible upper and lower bounds on the worst case error for the function space ${\werr P\fsu}$ by using well-known exponential functions.
Generally speaking, the explicit expression for the worst case error is not so easy to obtain
 for any point set $P$.
The case where the function space $F$ is a reproducing kernel Hilbert Space (RKHS) is the most preferable situation, e.g., see \cite[Chapter~2]{Dick_Pill} for detailed information of RKHS\@.
In this case the worst case error by a point set $P$ is presented by
\[
\werr PF = \aerr P{f_{F,P}},
\]
where $f_{F,P}$ depends on both of $F$ and the choice of $P$.
This identity implies that there is some function $f_{F,P}$, whose integration error attains the worst case error by $P$.
Thus, in analyzing the worst case error, we have to replace $f_{F,P}$ every time we replace $P$.

In our case, the target function space $\fsu$ is not even a RKHS\@.
Nevertheless, we can obtain the following simple estimation for $\werr P\fsu$:
\[
L_{s,\boldsymbol{u}}\le \frac{\mathrm{Err}(\exp(-\sum_{j=1}^su_jx_j);P)}
{{\werr P\fsu}}\le U_{s,\boldsymbol{u}},
\]
for any digital net $P$ and some constants $L_{s,\boldsymbol{u}}$ and $U_{s,\boldsymbol{u}}$ depending on $s$ and $\boldsymbol{u}$ but not $P$ (see \refthm{e=Err} for the detailed description of the main theorem).
Although this is not an equality for the worst case error and we restrict the range of $P$ to the class of digital nets, this gives us with a simple interpretation for the integration error.
Moreover, the choice of integrands is independent of $P$ unlike the above case on a RKHS\@.

This result also has an important role when we try to find a digital net $P$ whose worst case error is small.
As we mentioned above, a computable upper bound $W_{\boldsymbol{u}}(P)$ on the worst case error for $\fsu$ is established, which is a generalization of WAFOM (see \refdfn{WAFOM} for $W_{\boldsymbol{u}}(P)$).
Using this quantity $W_{\boldsymbol{u}}(P)$, we could execute computer search for finding efficient QMC rules.
On the other hand, we may also use $\mathrm{Err}(\exp(-\sum_{j=1}^su_jx_j);P)$ instead of $W_{\boldsymbol{u}}(P)$, since this value $\mathrm{Err}(\exp(-\sum_{j=1}^su_jx_j);P)$ is also computable in reasonable time.
We compare these two quantities in \refsec{comparison-qualities}.
From the discussion there, when we calculate in a wide range of the weights $\boldsymbol{u}$ or higher dimensions $s$ under the condition that $u_j=u_1$ for $1\le j\le s$, we can expect that $\mathrm{Err}(\exp(-\sum_{j=1}^su_jx_j);P)$ is suitable for finding good QMC points.

In the proof of our main result, we also prove that $W_{\boldsymbol{u}}(P)$ approximates the worst case error as well as $\mathrm{Err}(\exp(-\sum_{j=1}^su_jx_j);P)$:
\[
L'_{s,\boldsymbol{u}}\le \frac{W_{\boldsymbol{u}}(P)}
{\werr P\fsu} \le U'_{s,\boldsymbol{u}},
\]
where $L'_{s,\boldsymbol{u}}$ and $U'_{s,\boldsymbol{u}}$ depends on not $P$ but $s$ and $\boldsymbol{u}$ (see \refcor{e=wf} for the explicit statement).
This result verifies that these two criteria have essentially the same role in the QMC error analysis.

The remainder of this article is organized as follows.
In \refsec{preliminaries} we recall some definitions for QMC integration by digital nets, and relation with Walsh coefficients.
In \refsec{proof-of-main-theorem} we prove the main result.
In \refsec{comparison-qualities}, we show the numerical properties on $\serr Pf$ and $W_{\boldsymbol{u}}(P)$ in terms of running times and the ratio of these two quantities.

\section{Preliminaries}
\lblsec{preliminaries}
We denote $\N=\mathbb{N}\cup\{0\}$ in this paper.
In this section we briefly recall the notion of digital nets and Walsh coefficients.
\subsection{QMC integration by digital nets}
We first introduce digital nets over the two-element field $\B=\{0,1\}$,
which is defined as follows.
\begin{dfn}[digital net over $\B$ \cite{Nied}]
\lbldfn{digital-net}
Let $n,m\ge 1$ be integers with $n\ge m$.
Let $0\le h< 2^{m}$ be an integer and $C_1,\dots,C_s$ be $n\times m$ matrices over the finite field $\B$.
We write the dyadic expansion $h=\sum_{i=1}^mh_i2^{i-1}$ and take a vector $\boldsymbol{h}=(h_1,\dots,h_m)\in(\B^m)^\top$,
where $h_i$ is considered to be an element in $\B$.
For $1\le j\le s$, we define the vector
\[(y_{h,1,j},\dots,y_{h,n,j})=\boldsymbol{h}\cdot(C_j)^{\top}\]
and a real number
\[x_j(h)=\sum_{1\le i\le n}y_{h,i,j}2^{-i}\in\I,\]
where $y_{h,i,j}$ is considered to be an element of $\{0,1\}\subset\Z$.
Then we define a digital net $P$ by $\{\boldsymbol{x}_0,\cdots,\boldsymbol{x}_{2^m-1}\}$ where $\boldsymbol{x}_h=(x_j(h))_{1\le j\le s}$.
\end{dfn}

\subsection{Error analysis via Walsh coefficients}
We express QMC integration error for functions in the normed function space $\fsu$ through Walsh coefficients, which is introduced using Walsh functions as follows.
\begin{dfn}[Walsh functions and Walsh coefficients over $\B$]
Let $f\colon\I^s\rightarrow \R$ and $\boldsymbol{k}=(k_1,\dots,k_s)\in\N^s$.
We define the $\boldsymbol{k}$-th dyadic Walsh function
$\mathrm{wal}_{\boldsymbol{k}}\colon \I \to \{\pm 1\}$ by
\begin{align*}
\mathrm{wal}_{\boldsymbol{k}}(\boldsymbol{x}):=\prod_{j=1}^s(-1)^{\sum_{i\ge 1}\kappa_{i,j}x _{i,j}}.
\end{align*}
Here we write the dyadic expansion of $k_j$ by $k_j=\sum_{i\ge 1}\kappa_{i,j} 2^{i-1}$ and $x_j$ by $x_j=\sum_{i\ge 1}x_{i,j}2^{-i}$, where infinitely many digits $x_{i,j}$ are $0$.

Using Walsh functions, we define the $\boldsymbol{k}$-th dyadic Walsh coefficient $\widehat{f}(\boldsymbol{k})$ as
\begin{align*}
\widehat{f}(\boldsymbol{k})
:=\int_{\I^s}f(\boldsymbol{x})\cdot\mathrm{wal}_{\boldsymbol{k}}(\boldsymbol{x})\, d\boldsymbol{x}.
\end{align*}
\end{dfn}
For general information on Walsh analysis see \cite{Fine,SWS} for example.
To exploit the linear structure of digital nets, we need the notion of dual nets \cite{Dick_Pill05,NiedP}.
\begin{dfn}
[dual net over $\B$ \cite{NiedP}]
We define the dual net of a digital net $P$ by
\begin{eqnarray*}
P^\bot
 :=\{\mathbf{\boldsymbol{k}}=(k_1,\dots, k_s)\in\N^s\mid
 C_1^{\top}\vec{k}_1+\cdots+C_s^{\top}\vec{k}_s=\mathbf{0}\in\B^m\},
\end{eqnarray*}
where $\vec{k}_j=(\kappa _{1,j},\dots,\kappa _{n,j})^{\top}$ for
$k_j$ with $b$-adic expansion $k_j=\sum_{i\ge 1}\kappa _{i,j}2^{i-1}$.
\end{dfn}
The following is the key lemma in analyzing QMC integration error by digital nets
(see \cite[Chapter~15]{Dick_Pill}, \cite[Lemma~18]{GSY}):
\begin{prop}
\lblprp{ErrW}
Let $f \in C^0\Ic^s$ with $\sum_{\boldsymbol{k} \in \N^s}\abs{\widehat{f}(\boldsymbol{k})} < +\infty$ and $P\subset\I^s$ be a digital net.
The QMC integration error of $f$ by $P$ is given by
\[ \serr Pf = \sum_{\boldsymbol{k} \in P^\bot\setminus\{\boldsymbol{0}\}}\widehat{f}(\boldsymbol{k}). \]
\end{prop}
\begin{rmk}
If $f$ has continuous derivatives $f^{(\alpha_1,\dots,\alpha_s)}$ for any nonnegative integers $\alpha_j\le 2$,
it holds that $\sum_{\boldsymbol{k} \in \N^s}\abs{\widehat{f}(\boldsymbol{k})} < +\infty$  (see \cite[Section~2]{Ywalupp}).
Note that any function in the function class $\fsu$ satisfies this assumption.
\end{rmk}

\section{Worst case error for weighted smooth function spaces}
\lblsec{proof-of-main-theorem}
In this section, we give the main theorem, which shows a simple approximation of
the worst case error ${\werr P\fsu}$.
\begin{thm}
\lblthm{e=Err}
Let $\boldsymbol{u}=(u_j)_{j=1}^s\in\R^s$ be positive real numbers and
$g_{s,\boldsymbol{u}}(\boldsymbol{x}):=\exp(-\sum_{j=1}^su_jx_j)$.
For any digital net $P$, we have
\begin{align}
\lbleq{main-result-1}
A_{s,\boldsymbol{u}}^{-1}\cdot \mathrm{Err}(g_{s,\boldsymbol{u}};P)
\le
{\werr P\fsu}
\le
2^sB_{s,\boldsymbol{u}}^{-1}\cdot \mathrm{Err}(g_{s,\boldsymbol{u}};P),
\end{align}
where
\begin{align*}
A_{s,\boldsymbol{u}}&=\prod_{j=1}^s\frac{1-\exp(-u_j)}{u_j}, \\
B_{s,\boldsymbol{u}}&=
\prod_{j=1}^s\inf_{1 \leq w}
\left(\frac{1-\exp(-2^{-w}u_j)}{2^{-w}u_j}\prod_{1\leq i \leq w}\frac{1-\exp(-2^{-i}u_j)}{2^{-i}u_j}\right)\neq 0.
\end{align*}
In particular if we choose the normalized function $\tilde{g}_{s,\boldsymbol{u}}:=g_{s,\boldsymbol{u}}(x)/\|g_{s,\boldsymbol{u}}\|_{\fsu}$, we have
\begin{align}
\lbleq{main-result-2}
\mathrm{Err}(\tilde{g}_{s,\boldsymbol{u}};P)
\le
{\werr P\fsu}
\le
\frac{2^sA_{s,\boldsymbol{u}}}{B_{s,\boldsymbol{u}}}\cdot \mathrm{Err}(\tilde{g}_{s,\boldsymbol{u}};P).
\end{align}
\end{thm}
\begin{rmk}
This theorem implies that the integration error of the exponential function $g_{s,\boldsymbol{u}}$ by a digital net is always positive.
\end{rmk}

Now we start the proof for this theorem.
We have only to prove the first statement.
In fact, the second statement \refeq{main-result-2} follows from the first statement \refeq{main-result-1} since $\|g_{s,\boldsymbol{u}}\|_{\fsu}= \prod_{j=1}^s(1-\exp(-u_j))/u_j=A_{s,\boldsymbol{u}}$ and thus
\begin{align}
A^{-1}_{s,\boldsymbol{u}}\cdot \mathrm{Err}\left(g_{s,\boldsymbol{u}};P\right)
=\mathrm{Err}\left(\tilde{g}_{s,\boldsymbol{u}};P\right).
\label{normalize}
\end{align}

In what follows, we provide the proof of \refeq{main-result-1}.

The left inequality in \refeq{main-result-1} is followed by (\ref{normalize}) and the definition of ${\werr P\fsu}$.

We move on to the proof on the right inequality in \refeq{main-result-1}.
This is proven by using \reflem{ewor<wf} and \ref{lem:wf<exp}.
Before starting the proof, we introduce the following criteria on a digital net $P$.
\begin{dfn}
\lbldfn{WAFOM}
Let $P\subset \I^s$ be a digital net.
For a positive real numbers $\boldsymbol{u}\in\R^s$, we define the following function of $P$
\[ W_{\boldsymbol{u}}(P) := \sum_{\boldsymbol{k} \in P^\perp \smallsetminus\{\boldsymbol{0}\}}2^{-\mu_{\boldsymbol{u}}(\boldsymbol{k})}, \]
where for $\boldsymbol{k}=(k_1,\dots,k_s)$ with $k_j= \sum_{i\ge 1}\kappa_{i,j}2^{i-1}\in \N$, we define
\[\mu_{\boldsymbol{u}}(\boldsymbol{k}) = \sum_{1 \leq j \leq s}\sum_{1 \leq i}(i+1-\lg u_j)\kappa_{i,j}.\]
\end{dfn}
When $k_j=0$, we set $\kappa_{i,j}=0$ for all $i$.
\begin{rmk}
This quantity is introduced in \cite{K:acc} (see \cite[Remark~6.4]{K:acc} for the details).
If we set $u_j=2$ for $1\le j\le s$, $W_{\boldsymbol{u}}(P)$ corresponds with the original WAFOM treated in \cite{MSM}.
\end{rmk}
We give the upper bound on the worst case error by this function $W_{\boldsymbol{u}}(P)$.
\begin{lem}
\lbllem{ewor<wf}
For any $f\in \fsu$ and a digital net $P$, we have
\begin{align}
{\serr Pf} \le 2^s\cdot \|f\|_{\fsu} \cdot  W_{\boldsymbol{u}}(P).
\end{align}
In particular, this inequality implies that
\begin{align}
{\werr P\fsu}\le 2^s\cdot W_{s,\boldsymbol{u}}(P).
\end{align}
\end{lem}
\begin{proof}
In \refprp{ErrW}, the QMC integration error by a digital net $P$ is rewritten
as a sum of Walsh coefficients.
Using the upper bounds on Walsh coefficients in \cite[Theorem~3.9]{SY} or \cite[Theorem~1.3]{Ywalupp}, it holds that for a function $f\in C^\infty\Ic^s$,
\[ \abs{\widehat{f}(\boldsymbol{k})} \leq 2^{s}\cdot 2^{-\mu_{(1,\dots,1)}(\boldsymbol{k})}\cdot \|f^{(\alpha_1,\dots,\alpha_s)}\|_{L^1}.\]
where $\boldsymbol{k}=(k_1,\dots,k_s)$ with $k_j=\sum_{i\ge 1}\kappa_{i,j}2^{i-1}$
and $\alpha_j=|\{i\ge 1\mid \kappa_{i,j}=1\}|$.
In this condition, we have for $\boldsymbol{u}=(u_1,\dots,u_s)$ with $u_j>0$,
\begin{align*}
\mu_{\boldsymbol{u}}(\boldsymbol{k})
&=\sum_{1 \leq j \leq s}\sum_{1 \leq i}(i+1-\lg u_j)\kappa_{i,j}\\
&=\sum_{1 \leq j \leq s}\sum_{1 \leq i}(i+1)\kappa_{i,j}
-\sum_{1 \leq j \leq s}\sum_{1 \leq i}(\lg u_j)\kappa_{i,j}\\
&=\mu_{(1,\dots,1)}(\boldsymbol{k})
-\sum_{1 \leq j \leq s}\lg u^{\alpha_j}_j.
\end{align*}
Thus, for a function $f\in\fsu$, we have
\[ \abs{\widehat{f}(\boldsymbol{k})} \leq \|f\|_{\fsu}\cdot 2^{s}\cdot 2^{-\mu_{\boldsymbol{u}}(\boldsymbol{k})}. \]
Combining this upper bound and \refprp{ErrW}, we obtain the following evaluation on the QMC error by a digital net $P$.
\[
{\serr Pf}\le 2^s\cdot \|f\|_{\fsu}
\cdot \sum_{k \in P^\perp \smallsetminus\{0\}}2^{-\mu_{\boldsymbol{u}}(k)}
=2^s\cdot \|f\|_{\fsu}\cdot  W_{\boldsymbol{u}}(P).
\]
The second inequality follows from this inequality as
\[
{\werr P\fsu}
=\sup_{\substack{f\in \fsu\\ \|f\|_{\fsu}\le 1}}{\aerr Pf}
\le 2^s\cdot1\cdot W_{\boldsymbol{u}}(P).
\]
\end{proof}
We next obtain the upper bound on $W_{\boldsymbol{u}}(P)$ by $\serr Pg$.
In fact, we can obtain both of lower and upper bounds as follows.
\begin{lem}
\lbllem{wf<exp}
For $g_{s,\boldsymbol{u}}:=\exp(-\sum_{1 \leq j \leq s}u_jx_j)$ and any digital net $P$,
it holds that
\[ A_{s,\boldsymbol{u}}^{-1}\cdot \serr P{g_{s,\boldsymbol{u}}}\le W_{\boldsymbol{u}}(P) \le B_{s,\boldsymbol{u}}^{-1}\cdot \serr P{g_{s,\boldsymbol{u}}}, \]
where $A_{s,\boldsymbol{u}}$ and $B_{s,\boldsymbol{u}}$ are the same as in \refthm{e=Err}.
\end{lem}

Using these lemmas, we can finish the proof of \refthm{e=Err}.
In fact, combining \reflem{ewor<wf} and \ref{lem:wf<exp}, we get the right inequality in \refeq{main-result-1} as follows.
\begin{align*}
{\werr P\fsu}
\le2^s\cdot  W_{\boldsymbol{u}}(P)
\le2^s\cdot  B_{s,\boldsymbol{u}}^{-1}\cdot \serr P{g_{s,\boldsymbol{u}}},
\end{align*}
which is the right inequality in \refeq{main-result-1}.

In the next subsection, we provide the proof of \reflem{wf<exp}.

\subsection{Approximation on $W_{\boldsymbol{u}}(P)$ by $\serr P{g_{s,\boldsymbol{u}}}$}
\reflem{wf<exp} is proven by calculating the Walsh coefficients of exponential functions.
We first consider the case $s=1$, in which we explicitly calculate the Walsh coefficients of exponential functions as follows.
\begin{lem}
\lbllem{wal-exp}
Let $a$ be a nonzero real number and $g \colon \I \to \R; x \mapsto \exp(ax)$.
Let $k$ be a nonnegative integer and $(\kappa_j)_{1 \leq j \leq w}$ be the binary representation with $\kappa_{w} = 1$, i.e., $k = \sum_{1 \leq i \leq w}2^{i-1}\kappa_i$.
If $k = 0$, we set $w = 0$ and interpret the sum to be empty.

Then, the $k$-th Walsh coefficient $\widehat{g}(k)$ of $g$ is calculated as
\[ \widehat{g}(k) = \frac{g(2^{-w})-1}{a}\prod_{1 \le i \leq w}(1+(-1)^{\kappa_i}g(2^{-i})). \]
\end{lem}
\begin{proof}
Since the case of $k = 0$ is easily proven by direct calculation, we assume $k > 0$ hereafter.

In the following calculation, we divide the interval $\I$ into $2^w$ intervals $\hoi{2^{-w}i}{2^{-w}(i+1)}$ for $0 \leq i < 2^w$, in each of which the $k$-th Walsh function remains constant.
We use the binary inverse expansion $(\zeta_j)_{j=1}^w$ of $l = 2^{w-1}\zeta_1 + \dots + 2^0\zeta_{w}$ which corresponds with the binary expansion of $x=\sum_{i\ge 1}\zeta_i2^{-i}$ if $x \in \hoi{2^{-w}l}{2^{-w}(l+1)}$.
The following computation proves the lemma:
\begin{align*}
\widehat{g}(k)
&= \int_0^1\wal k(x)\,\exp (ax)\,dx \\
&= \sum_{0 \leq l < 2^w}\wal k\mathord{\left(\frac l{2^w}\right)}\exp\left(\frac {a\cdot l}{2^w}\right)\int_0^{2^{-w}}\exp(ax)\,dx \\
&= \sum_{0 \leq l < 2^w}\prod_{1 \leq i \leq w}(-1)^{\kappa_i\zeta_i}\prod_{1
 \leq i \leq w}\exp\left(\frac{a\cdot2^{w-i}\cdot \zeta_i}{2^w}\right)\cdot \frac1a(\exp (2^{-w}a)-1) \\
&= \sum_{0 \leq l < 2^w}\prod_{1 \leq i \leq w}
\left((-1)^{\kappa_i}\exp\left(\frac a{2^i}\right)\right)^{\zeta_i}\cdot \frac1a(\exp (2^{-w}a)-1) \\
&= \frac{\exp (2^{-w}a)-1}{a}\sum_{(\zeta_i)_{i=1}^w \in \left\{0,1\right\}^{w}}\prod_{1 \leq i \leq w}\mathord{\left((-1)^{\kappa_i}\exp(2^{-i}a\right)})^{\zeta_i} \\
&= \frac{\exp (2^{-w}a)-1}{a}\prod_{1 \leq i \leq w}(1 + (-1)^{\kappa_i}\exp (2^{-i}a))\\
&= \frac{g(2^{-w})-1}{a}\prod_{1 \leq i \leq w}(1 + (-1)^{\kappa_i}g(2^{-i})). \qedhere
\end{align*}
\end{proof}
Using this lemma, we give the upper and lower bound on the Walsh coefficients of
$\exp(-ux)$ with $u>0$.
\begin{lem}
\lbllem{wal-bound}
Let $u>0$ and $g_u \colon \I \to \R; x \mapsto \exp(-ux)$.
It holds that
\[ B_u 2^{-\mu_u(k)}\leq \widehat{g}(k) \leq A_u 2^{-\mu_u(k)}\]
for all nonnegative $k$, where
$A_{u},B_{u}$ equals $A_{1,u},B_{1,u}$ appearing in \refthm{e=Err}, respectively.
If $k=0$, we have
\[
A_u 2^{-\mu_u(0)}=\widehat{g}(0).
\]
\end{lem}
\begin{proof}
We use the symbol in the proof of \reflem{wal-exp}.
In the case $k=0$, it follows from the above theorem directly.

We assume that $k\neq 0$.
Let $I := \{i \leq w \mid \kappa_i = 1\}$ and $J := \{i < w \mid \kappa_i = 0\}$.
From \reflem{wal-exp} we have
\begin{align*}
\widehat{g_u}(k)
&= \frac{1-g_u(2^{-w})}{u}\prod_{1 \leq i \le w}\left(1+(-1)^{\kappa_i}g_u(2^{-i})\right) \\
&= \frac{1-g_u(2^{-w})}{u}\cdot\prod_{i \in I}\left(1-g_u(2^{-i})\right)\cdot\prod_{i \in J}\left(1+g_u(2^{-i})\right) \\
&= 2^{-w}\cdot \frac{1-g_u(2^{-w})}{2^{-w}u}\cdot
\prod_{i \in I}2^{-i}u\frac{1-g_u(2^{-i})}{2^{-i}u}\cdot
\prod_{i \in J}2\frac{1+g_u(2^{-i})}2\\
&= \left(2^{-w}\cdot\prod_{i \in I}2^{-i}u\cdot\prod_{i \in J}2\right)
\cdot\left(\frac{1-g_u(2^{-w})}{2^{-w}u}\cdot\prod_{i \in I}\frac{1-g_u(2^{-i})}{2^{-i}u}\cdot\prod_{i \in J}\frac{1+g_u(2^{-i})}2\right).
\end{align*}
The former term can be calculated as follows.
\begin{align*}
2^{-w}\cdot\prod_{i \in I}2^{-i}u\cdot\prod_{i \in J}2
&= \prod_{i \in I}2^{-i-1}u\prod_{j \in J}1\\
&= \prod_{i=1}^w2^{-(i+1-\lg u)\kappa_i}
= 2^{-\mu_u(k)}.
\end{align*}
Thus, it suffices to give bounds of the form
\begin{equation}
\lbleq{forLU}
B_u \le
\frac{1-g_u(2^{-w})}{2^{-w}u}\cdot\prod_{i \in I}\frac{1-g_u(2^{-i})}{2^{-i}u}\cdot\prod_{i \in J}\frac{1+g_u(2^{-i})}2
\le A_u.
\end{equation}

For fixed $w$, i.e., if $2^{w-1} \leq k < 2^{w}$, it holds that
\begin{equation}
\lbleq{ineq1}
\prod_{1 \leq i \leq w}\frac{1-g(2^{-i})}{2^{-i}u}
  \leq
\prod_{i \in I}\frac{1-g_u(2^{-i})}{2^{-i}u}\prod_{j \in J}\frac{1+g_u(2^{-i})}2
  \leq
\prod_{1 \leq i \leq w}\frac{1+g_u(2^{-i})}2,
\end{equation}
where we use the fact that $(1-g_u(x))/ux=-g_u'(y_x)<(1+g_u(x))/2$ holds for every $x>0$ and some $0< y_x< ux$.
The right hand side in \refeq{ineq1} is calculated as
\begin{align*}
\prod_{1 \leq i \leq w}\frac{1+g(2^{-i})}2
&= 2^{-w}\prod_{1 \leq i \leq w}(1+\exp(-2^{-i}u))
= 2^{-w}\prod_{1 \leq i \leq w}(1+\exp(-2^{-w}u\cdot2^{i-1})) \\
&= 2^{-w}\frac{1-\exp(-2^{-w}u\cdot 2^w)}{1-\exp(-2^{-w}u)}
= \frac{2^{-w}}{1-g_u(-2^{-w})}(1-\exp(-u)).
\end{align*}
By inserting this in the right term of \refeq{ineq1} and multiplying
$(1-g_u(2^{-w}))/2^{-w}u$ we obtain
\begin{align*}
&\mathrel{\phantom\le}\frac{1-g_u(2^{-w})}{2^{-w}u}\cdot\prod_{i \in I}\frac{1-g_u(2^{-i})}{2^{-i}u}\cdot\prod_{i \in J}\frac{1+g_u(2^{-i})}2\\
&\le\frac{1-g_u(2^{-w})}{2^{-w}u}\frac{2^{-w}}{1-g_u(-2^{-w})}(1-\exp(-u))
=  \frac{1-\exp(-u)}{u}.
\end{align*}
Thus, $A_u := (1-\exp(-u))/u$ satisfies the upper bound in \refeq{forLU}.

Next we move on to the lower bound of \refeq{forLU}.
Since $(1-\exp(-x))/x < 1$ holds for all $x>0$ and
$(1-g_u(2^{-w}))/(2^{w}u)$ goes to $1$ as $w$ goes to infinity,
it is enough to show that the infinite product
\[ \prod_{1\leq i}\frac{1-g_u(2^{-i})}{2^{-i}u} \]
converges into some nonzero value.

Since $\exp(-x) < (1-\exp(-x))/x < 1$ for all positive $x$,
\[ \prod_{1 \leq i}\exp(-2^{-i}u) \leq \prod_{1 \leq i}\frac{1-g_u(2^{-i})}{2^{-i}u} \leq 1 \]
holds, whose left hand side converges since $\sum_{1\leq i}-2^{-i}u=-u$.
Thus there exists
\[0 < B_u := \inf_{1 \leq w}\left(\frac{1-g_u(2^{-w})}{2^{-w}u}\prod_{1\leq i \leq w}\frac{1-g_u(2^{-i})}{2^{-i}u}\right) \]
which satisfies the claim.
\end{proof}
\begin{rmk}
For $u=2$, which corresponds to the case of the one-dimensional original WAFOM, the value of $B_2$ and $A_2$ are approximately $0.388$ and $0.432$.
\end{rmk}

Next we give bounds on the Walsh coefficients of $\exp(-\sum_{j=1}u_jx_j)$,
which is easily calculated using the above one-dimensional result.
\begin{lem}
\lbllem{wal-bound-multi}
Let $(u_j)_{j=1}^s\in\mathbb{R}^s$ be positive and $g_{s,\boldsymbol{u}}(\boldsymbol{x})=\exp(-\sum_{j=1}^su_jx_j)$.
It holds that
\[ B_{s,\boldsymbol{u}} 2^{-\mu_{\boldsymbol{u}}(\boldsymbol{k})}\leq \widehat{g}_{s,\boldsymbol{u}}(\boldsymbol{u}) \leq A_{s,\boldsymbol{u}} 2^{-\mu_{\boldsymbol{u}}(\boldsymbol{k})}\]
for all $\boldsymbol{k}\in\N^s$, where
$A_{s,\boldsymbol{u}}$ and $B_{s,\boldsymbol{u}}$ are the same as in \refthm{e=Err}.
\end{lem}
\begin{proof}
Since $g_{s,\boldsymbol{u}}(\boldsymbol{x})$ and $\mathrm{wal}_{\boldsymbol{k}}(\boldsymbol{x})$ is written by the product of univariate functions
$\exp(-u_jx_j)$ and $\mathrm{wal}_{k_j}(x_j)$, respectively, we have
\begin{align*}
\widehat{g}_{s,\boldsymbol{u}}(\boldsymbol{k})
&=\int_{\I^s}\prod_{j=1}^s\exp(-u_jx_j)\mathrm{wal}_{k_j}(x_j)\,dx_1\dots dx_s\\
&=\prod_{j=1}^s\int_{\I}\exp(-u_jx_j)\mathrm{wal}_{k_j}(x_j)\,dx_j
=\prod_{j=1}^s\widehat{g}_{u_j}(k_j).
\end{align*}
Combining this with \reflem{wal-bound} we get
\begin{align*}
B_{s,\boldsymbol{u}}\cdot 2^{-\mu_{\boldsymbol{u}}(\boldsymbol{k})}
=
\prod_{j=1}^sB_{u_j}2^{-\mu_{u_j}(k_j)}
\le
\widehat{g}_{s,\boldsymbol{u}}(\boldsymbol{k})
\le
\prod_{j=1}^sA_{u_j}2^{-\mu_{u_j}(k_j)}
=
A_{s,\boldsymbol{u}}\cdot 2^{-\mu_{\boldsymbol{u}}(\boldsymbol{k})}.
\end{align*}
\end{proof}

This inequality and \refprp{ErrW} leads to \reflem{wf<exp} as
\begin{align*}
B_{s,\boldsymbol{u}}\cdot
W_{\boldsymbol{u}}(P)
\le
\mathrm{Err}(g_{s,\boldsymbol{u}};P)=\sum_{\boldsymbol{k}\in P^\bot\setminus\{0\}}\widehat{g}_{s,\boldsymbol{u}}(k)
\le
A_{s,\boldsymbol{u}}\cdot
W_{\boldsymbol{u}}(P).
\end{align*}

\subsection{Approximation for the worst case error by the quantity $W_{s,\boldsymbol{u}}(P)$}
As we mentioned in the last part of \refsec{introduction}, $\werr P\fsu$ is also bounded below and above by $W_{\boldsymbol{u}}(P)$ as follows.
\begin{cor}
\lblcor{e=wf}
Let $\boldsymbol{u}=(u_j)_{j=1}^s$ be positive real numbers.
For any digital net $P$, we have
\begin{align*}
\frac{A_{s,\boldsymbol{u}}}{B_{s,\boldsymbol{u}}}\cdot W_{\boldsymbol{u}}(P)
\le
{\werr P\fsu}
\le
2^s\cdot W_{\boldsymbol{u}}(P),
\end{align*}
where $A_{s,\boldsymbol{u}}$ and $B_{s,\boldsymbol{u}}$ are the same as in \refthm{e=Err}.
\end{cor}
\begin{proof}
The left inequality is deduced as
\begin{align*}
\frac{A_{s,\boldsymbol{u}}}{B_{s,\boldsymbol{u}}}\cdot W_{\boldsymbol{u}}(P)
\le
A_{s,\boldsymbol{u}}\mathrm{Err}(g_{s,\boldsymbol{u}};P)
\le
{\werr P\fsu},
\end{align*}
where we used \reflem{wf<exp} in the first inequality and \refthm{e=Err} in the second one.
The right inequality
directly follows from \reflem{ewor<wf}.
\end{proof}
Moreover, we have the following formula as in the case of the original WAFOM\@.
\begin{cor}
\lblcor{wf-calc-formula}
For any digital net $P$ with $N=|P|$, we have
\[
W_{\boldsymbol{u}}(P) =
\frac{1}{N}
\sum_{\boldsymbol{x}\in P}
\prod_{1\leq j \leq s}\prod_{1\leq i}\left(1+(-1)^{\zeta_{i,j}}2^{-(i+1-\lg u_j)}\right)-1,
\]
where we write $\boldsymbol{x} = (\sum_{i\ge 1}2^{i-1}\zeta_{i,j})$.
\end{cor}
We omit the proof, which is almost the same as in the case of the original WAFOM~\cite[Corollary~4.2]{MSM}.
\begin{rmk}
\lblrmk{wf-calc-formula}
$ $
\begin{itemize}
\item
To calculate this value on computer, we consider the following discretized version:
\[
W_{\boldsymbol{u}}^{n}(P) =\frac{1}{N}
\sum_{\boldsymbol{x}\in P}
\prod_{1\leq j \leq s}\prod_{1\leq i\le n}\left(1+(-1)^{\zeta_{i,j}}2^{-(i+1-\lg u_j)}\right)-1,
\]
where we write $\boldsymbol{x} = (\sum_{i\ge 1}2^{i-1}\zeta_{i,j})$.
In fact, this discretized version in the case $u_j=2$ is originally introduced in \cite[Definition~3.5]{MSM} as a definition of WAFOM\@.
They introduced this value as a computable quantity on digital nets for QMC integration error.
\item
This computable formula appears in \cite[Proposition~3]{DGSY} for so-called `polynomial lattice rules', which is a special type of digital nets.
\end{itemize}
\end{rmk}

\section{Practical Comparison}
\lblsec{comparison-qualities}
In previous sections we have shown that the following are approximately equivalent up to constants (see \refthm{e=Err}, \reflem{wf<exp} and \refcor{e=wf}):
\begin{enumerate}
    \item the worst-case integration error of the function space $\fsu$ by $P$,
    \item the quantity $W_{\boldsymbol{u}}$ of $P$, and
    \item the integration error of the exponential function $\mathrm{Err}(\exp(-\sum u_jx_j);P)$ by $P$.
\end{enumerate}
In this section we compare the second and third items from practical perspective.

\subsection{Implementation and Computational Complexity}
One advantage of the approximation $\mathrm{Err}(\exp(-\sum u_jx_j);P)$ is that its implementation is almost trivial.

Since na\"ive implementation of the formula in \refrmk{wf-calc-formula}
 for $W_{\boldsymbol{u}}^n(P)$ is slow, there have been some efforts to accelerate the computation.
Matsumoto, Saito and Matoba~\cite{MSM} restricted their search to sequential generators.
Due to the property of sequential generators, partial products can be reused when calculating the summands (they argue the case $u_j=2$).
The number of floating-point arithmetic operation reduces by a factor of $s$, which is an improvement in the order.
Another effort by Harase~\cite{H} (for the case that $u_j=2$, which is applicable for general $\boldsymbol u$) for accelerating the computation is to use lookup tables.
Though the timings will depend on many factors including code quality, compiler, optimization and CPU, his experiment shows a speedup by a factor of $30$.

Now we argue the computational complexity of ${\serr Pf}$ and $W_{\boldsymbol{u}}(P)$.
We consider the equal weight, i.e., all $u_j$ are equal (this setting includes the case of the original WAFOM).
Moreover, since $n$ can not go infinity in practical computation, we assume that $n$ is fixed.
In this case, the time complexity of his lookup-table method is dominated by $qs$ times multiplications, where $q$ is the number of tables and he used $q=3$ in his paper.
On the other hand, the time complexity of our approximation is dominated by $s$ times additions when $s$ is large.
If $s$ is small, the computational time of exponential functions determines the time complexity.
Thus we can expect that our method is faster than the method using look-up table on large dimensions $s$.

This is observed in the experiment.
In \cite[Table~1]{H}, we see the speed comparison between the calculation by look-up table and the integration error of the exponential function in the case $u_j=2$ for any $j$ and dimensions $s=2l$ with $2\le l \le 8$.
It shows that the computational times of both methods are proportional to $s$.
While our method is slower than his method, the increasing rate of our method per a dimension is smaller.
Thus we expect that our method gets faster when $s$ is large, say $s>30$.
Further, for the search for various values of the weights $\boldsymbol{u}$, lookup-table method requires precomputing for each $\boldsymbol{u}$ while our method does not.
Thus we expect that our approximation is suitable when a dimension $s$ is high or we vary $\boldsymbol{u}$ under the condition of equal weights.

\subsection{Approximation Accuracy}
As shown in \reflem{wf<exp}, the ratio of $\mathrm{Err}(\exp(-\sum_{j=1}^su_jx_j);P)$ to $W_{\boldsymbol{u}}(P)$ lies between $A_{s,\boldsymbol{u}}$ and $B_{s,\boldsymbol{u}}$.
While this is a theoretical bound, this subsection is to observe the distribution of the ratio through experiment.

The method of the experiment is as follows.
\begin{enumerate}
\item Fix the dimensionality $s$, the weights $\boldsymbol u$, the size of the digital net $m$, the number of precision digits $n$ and the number of trials $q$.
\item Generate $q$ digital nets $P$ by uniform random choice of their generating matrices.
\item Compute the integration error $\mathrm{Err}(\exp(-\sum_{j=1}^su_jx_j);P)$, WAFOM value $W_{\boldsymbol{u}}(P)$ and their ratio.
\end{enumerate}
We fix $u_j = 2$ (which corresponds with the original WAFOM), $n = 32$ and $q = 1024$.
\reffig{accuracy} shows the approximation accuracy for randomly generated digital nets with $(s,m) = (2,8)$, $(4,8)$, $(4,16)$ and $(8,16)$.
Here we use the notation $\boldsymbol{u}=\boldsymbol{2}$ meaning that $u_j=2$ for any $j$.

In each figure, the horizontal axis corresponds to $W_{\boldsymbol{2}}^{32}(P)$ and the vertical axis to the ratio $\serr P{\exp(-2\sum_{j=1}^s x_{j})}/W_{\boldsymbol{2}}^{32}(P)$, on both of which we use base-2 log-scale.
The horizontal lines represent the lower and upper bounds $B_{s,\boldsymbol{u}}$ and $A_{s,\boldsymbol{u}}$.
The value for $A_{s,\boldsymbol{u}}$ is computed by straightforward numerical approximation.
The value for $B_{s,\boldsymbol{u}}$ is computed by taking infimum for $w \leq n$.

\begin{figure}
\begin{minipage}{0.5\hsize}
\centering
\includegraphics[width=\hsize]{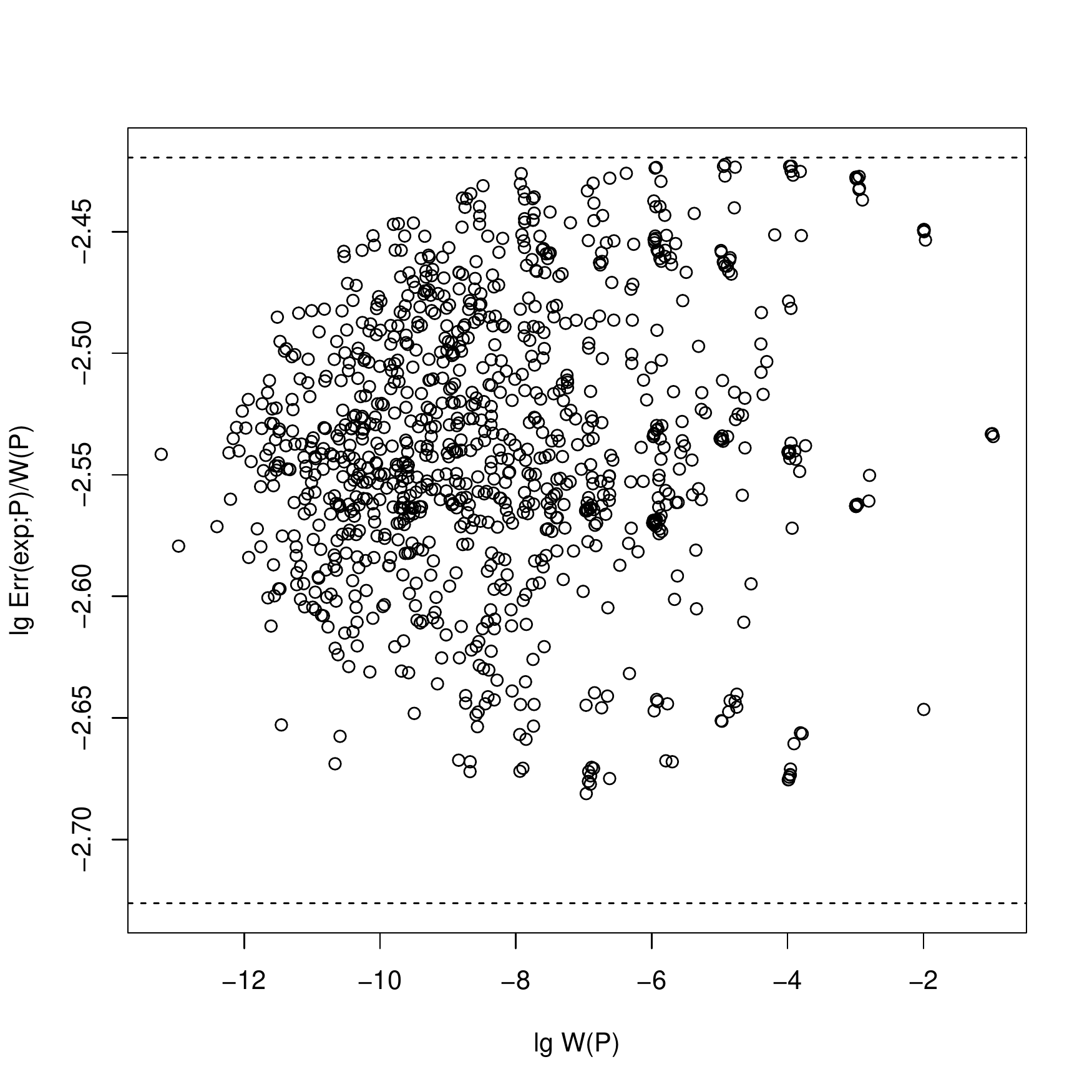}
\hspace{1cm}$(s, m) = (2, 8)$
\end{minipage}
\begin{minipage}{0.5\hsize}
\centering
\includegraphics[width=\hsize]{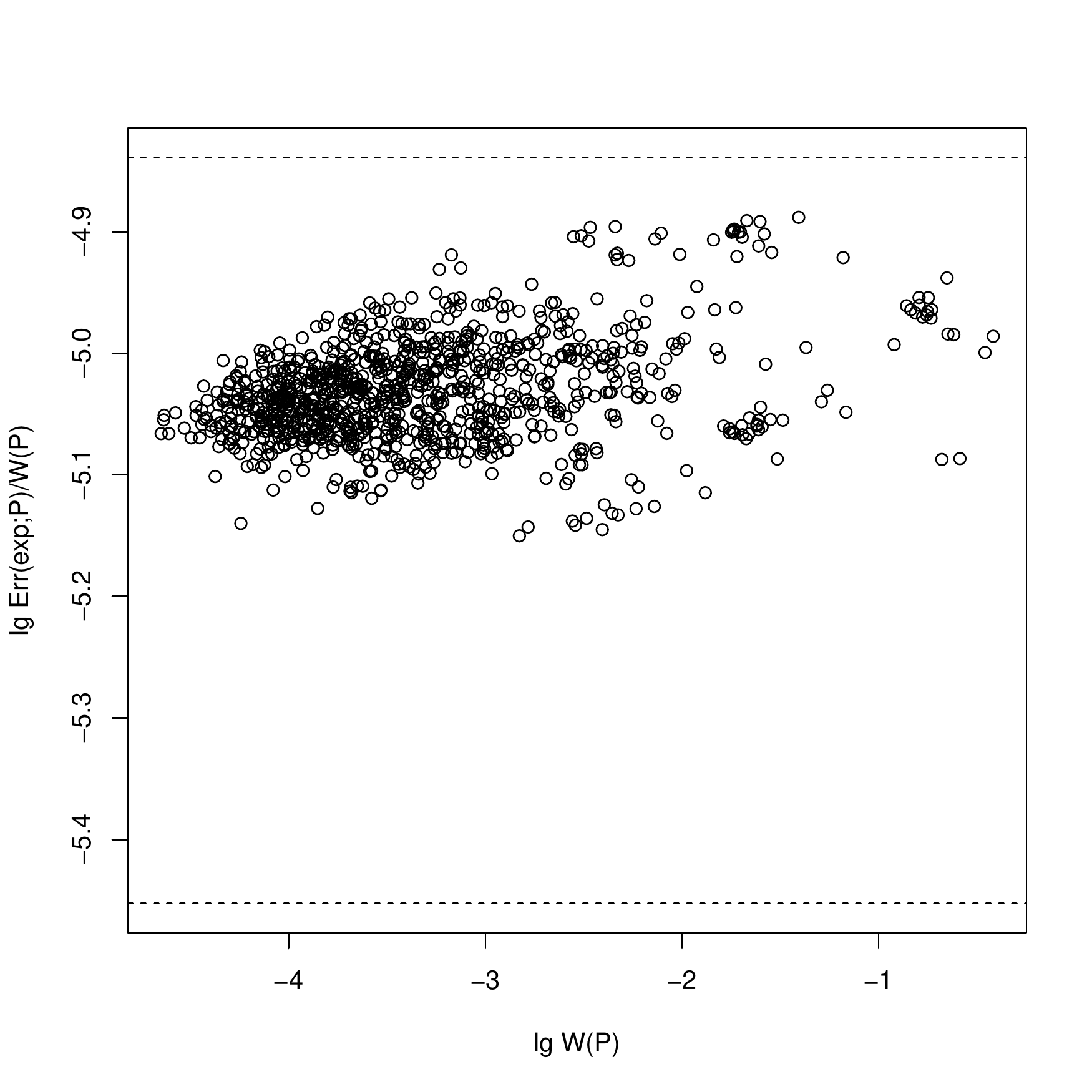}
\hspace{1cm}$(s, m) = (4, 8)$
\end{minipage}
\begin{minipage}{0.5\hsize}
\centering
\includegraphics[width=\hsize]{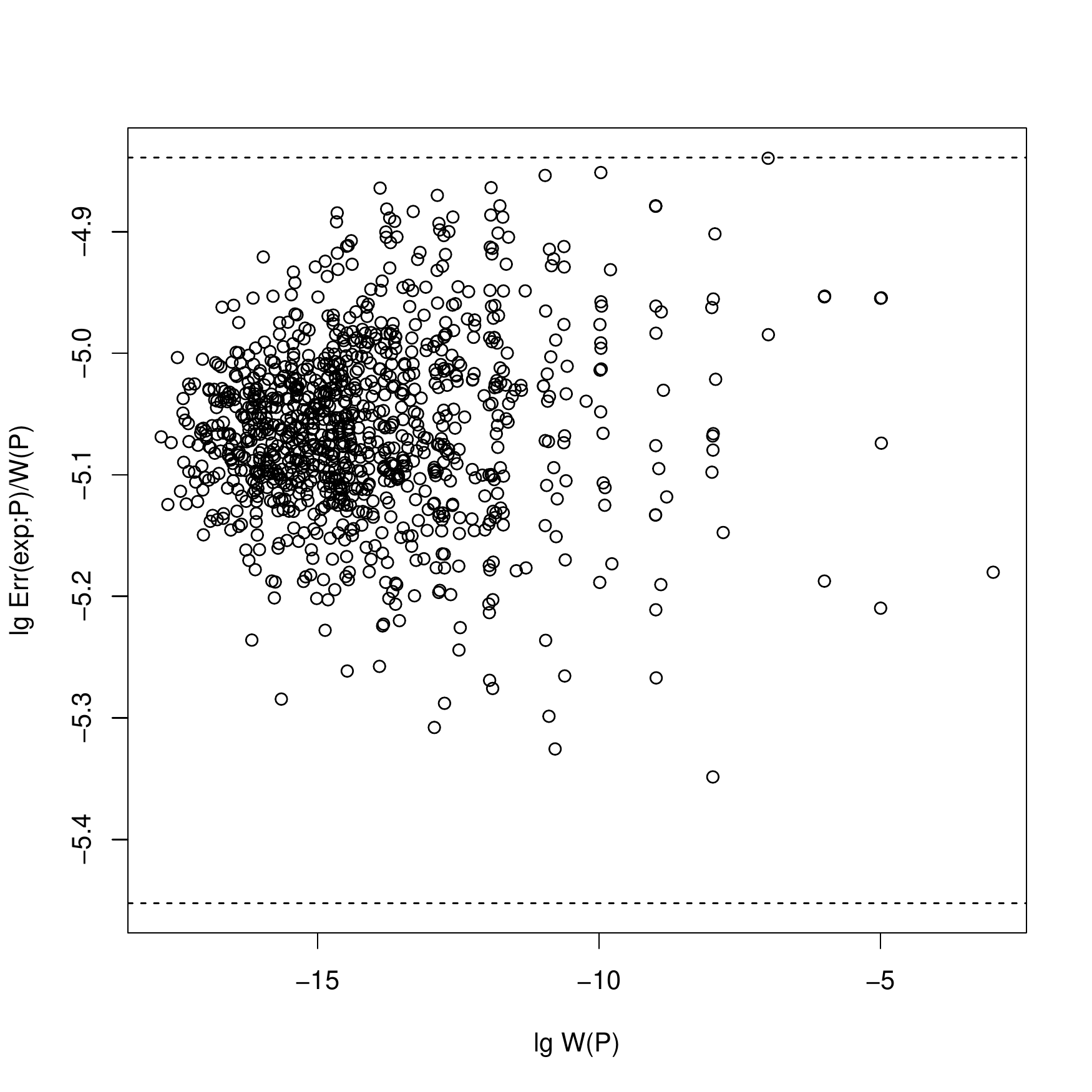}
\hspace{1cm}$(s, m) = (4, 16)$
\end{minipage}
\begin{minipage}{0.5\hsize}
\centering
\includegraphics[width=\hsize]{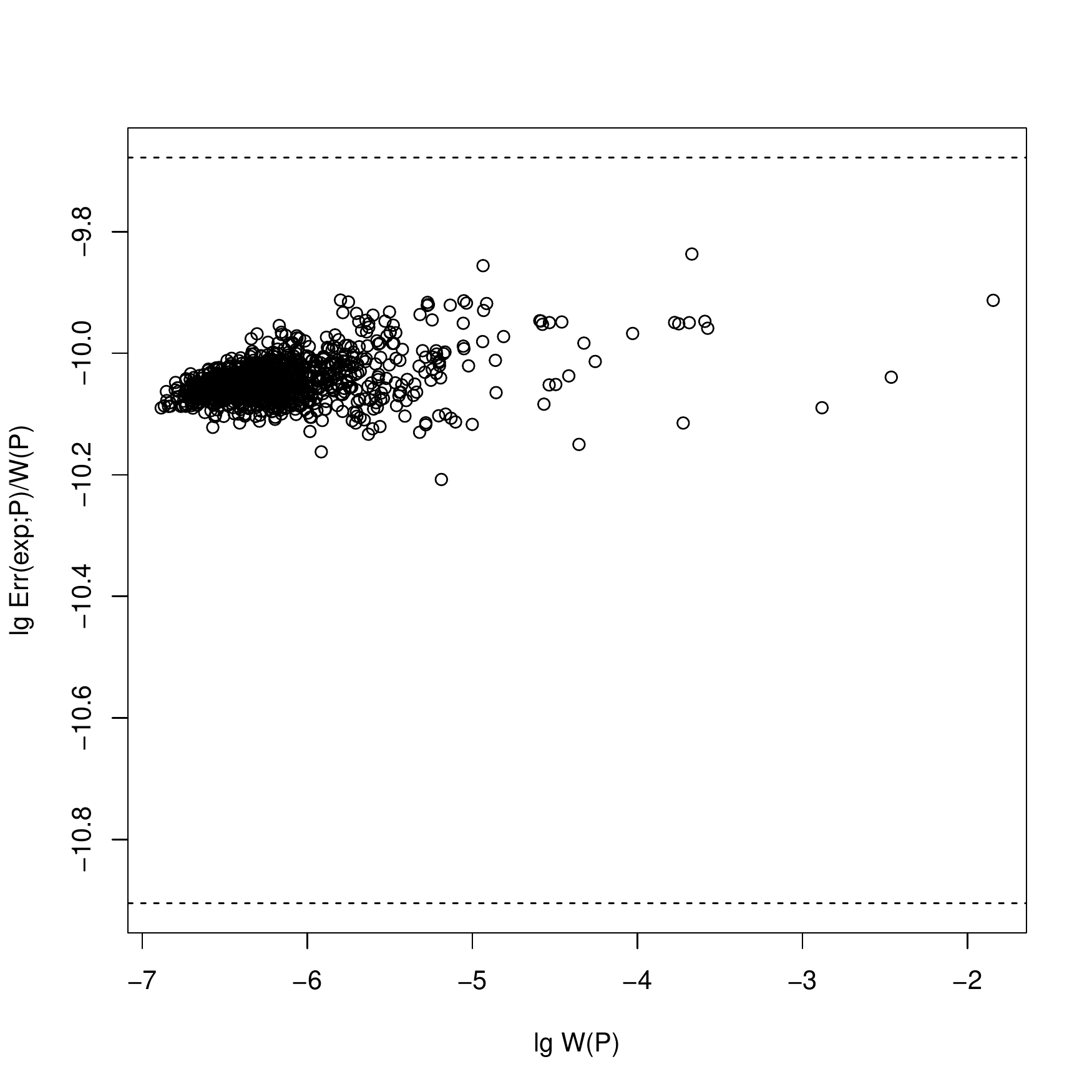}
\hspace{1cm}$(s, m) = (8, 16)$
\end{minipage}
\caption{The values $W_{\boldsymbol{2}}^{32}(P)$ and $\serr P{\exp(-2\sum_{j=1}^s x_{j})}/W_{\boldsymbol{2}}^{32}(P)$ for randomly generated digital nets $P$.}
\lblfig{accuracy}
\end{figure}

The figures show that the ratio is almost constant when $m$ is small and $s$ is large.
The ratio varies as $m$ increases.

\section{Conclusion}
We show that the QMC integration error of $\exp(-\sum_{1 \leq i \leq s} u_jx_j)$ by a digital net $P$ bounds the worst case error ${\werr P\fsu}$.
Here $\fsu$ is some smooth function class in $C^\infty\Ic^d$.
This fact implies that the exponential function is most difficult to approximate the integration value in this function space.
The combination of this fact and the simplicity of the exponential function leads to the efficient computer search of digital nets $P$ whose integration error $\serr Pf$ small for all $f\in \fsu$ uniformly.

\section{Acknowledgements}
The first author was supported by JSPS/MEXT Grant in Aid 23244002, 26310211, 15K13460.
The second and the third author were supported by the Program for Leading Graduate Schools, MEXT, Japan.
The first and third authors were supported by JST CREST\@.
The third author was supported under the Australian Research Councils Discovery Projects funding scheme (project number DP150101770).

\end{document}